\documentclass[reqno,10pt]{amsart}

\usepackage{amsmath}
\usepackage{amsfonts}
\usepackage{amssymb}
\usepackage{amsxtra}
\usepackage{amstext}
\usepackage{latexsym}
\usepackage{amsthm}
\usepackage{graphicx}
\usepackage[all,2cell]{xy}
\usepackage{dsfont} % for \mathds{N}
\usepackage{mathrsfs}
\usepackage{stmaryrd} %St. Mary Road Symbol. for right(left)lifting property.

\newtheorem{theorem}{Theorem}[section]
\newtheorem{lemma}[theorem]{Lemma}

\theoremstyle{definition}
\newtheorem{definition}[theorem]{Definition}

\newtheorem{def-thm}[theorem]{Definition-Theorem}
\newtheorem{def-lemma}[theorem]{Definition-Lemma}
\newtheorem{def-prop}[theorem]{Definition-Proposition}

\newtheorem*{xnotations}{Notation}

\theoremstyle{remark}
\newtheorem{remark}[theorem]{Remark}

\numberwithin{theorem}{section}
\numberwithin{equation}{section}
\numberwithin{figure}{section}
\numberwithin{table}{section}

\DeclareMathOperator{\ev}{ev}

\DeclareMathOperator{\id}{id}

\DeclareMathOperator{\op}{op}

\DeclareMathOperator{\pr}{pr}

\DeclareMathOperator{\Set}{\mathbf{Set}}

\DeclareMathOperator{\sSet}{\mathbf{sSet}}

\DeclareMathOperator{\ub}{\underline{\bullet}}
\DeclareMathOperator{\usSet}{\underline{\mathbf{sSet}}}

\newdir{ >}{{}*!/-10pt/@{>}}

\usepackage{eucal}

\renewcommand{\S}{\mathcal{S}}
\newcommand{\uS}{\underline{\mathcal{S}}}

\begin{document}

\title{Cylinders and paths in simplicial categories}

\author{Seunghun Lee}

\address{Department of Mathematics, Konkuk University,
Kwangjin-Gu Hwayang-dong 1, Seoul 143-701, Korea}

\email{mbrs@konkuk.ac.kr}

\thanks{This research was supported by Basic Science Research
Program through the National Research Foundation of Korea(NRF)
funded by the Ministry of Education, Science and
Technology(2009-0076403)}


\date{}

%\keywords{Algebraic geometry, multiplier ideals, integrally closed ideals}

\begin{abstract}
We prove the uniqueness, the functoriality and the naturality of
cylinder objects and path objects in closed simplicial model categories.
\end{abstract}

\maketitle

%\tableofcontents

%\onehalfspacing: for ohp_format.

\section{Introduction}

In \cite{quillen-67}, Quillen introduces the closed model category.
It is a category where one can do a homotopy theory.
In the first chapter, he sets up a general theory 
of closed model categories.
In the second chapter, he studies the closed simplicial model category.
They are closed model categories enriched in
the category $\sSet$ of simplicial sets with two additional properties.
Two basic notions in the closed model category are
the cylinder object and the path object.
They have no functorial properties in general.

Quillen also defines the cylinder object and the path object
in the closed simplicial model category.\footnote{Quillen did not name them.
But they play the same role as cylinder objects and path objects
in closed model categories do.
So we will call them the cylinder object and the path object.}
In fact, he defines them in the simplicial category.
We refer to Definition~\ref{def:esht.1.14}
for the definition of the map $\ev_{K,X}$.

\begin{definition}
\label{def:esht.1.2}
  Let $\S$ be a simplicial category.
  Let $X\in\S$ and $K\in\sSet$.
  \begin{enumerate}
    \item
        We call the following triple a cylinder object of $X$ over $K$,
        or simply a $K$-cylinder of $X$.
        \begin{enumerate}
          \item An object $X\otimes K$ of $\S$.
          \item
              A morphism of $\sSet$
              \begin{equation*}
                \alpha_{K,X}:K\rightarrow \uS(X,X\otimes K).
              \end{equation*}
          \item For every $Y\in \S$, an isomorphism of $\sSet$
              \begin{equation*}
                \phi_{K,X,Y}:\uS(X\otimes K,Y)
                  \rightarrow \usSet(K,\uS(X,Y))
              \end{equation*}
              making the following diagram commutes:
              \begin{equation*}
                \xymatrix{
                  K\times \uS(X\otimes K,Y)
                    \ar[d]_{\id_K\times \phi_{K,X,Y}}
                    \ar[rrr]^(0.44)
                      {\alpha_{K,X}\times_{\id_{\uS(X\otimes K,Y)}}}
                    &&& \uS(X,X\otimes K)\times \uS(X\otimes K , Y)
                    \ar[d]^{\ub_{X,X\otimes K ,Y}}\\
                  K\times\usSet(K, \uS(X,Y))
                    \ar[rrr]_(0.58){\ev_{K,\usSet(K, \uS(X,Y))}}
                      &&& \uS(X,Y)}
              \end{equation*}
        \end{enumerate}
    \item
        We call the following triple a path object of $X$ over $K$,
        or simply a $K$-path of $X$.
        \begin{enumerate}
          \item An object $X^K$ of $\S$.
          \item
              A morphism of $\sSet$
              \begin{equation*}
                \beta_{K,X}:K\rightarrow \uS(X^K,X).
              \end{equation*}
          \item For every $Y\in \S$, an isomorphism of $\sSet$
              \begin{equation*}
                \psi_{K,Y,X}:\uS(Y,X^K)
                  \rightarrow \usSet(K,\uS(Y,X))
              \end{equation*}
              making the following diagram commutes:
              \begin{equation*}
                \xymatrix{
                  K\times \uS(Y,X^K)
                    \ar[d]_{\id_K\times \psi_{K,Y,X}}
                    \ar[rrr]^(0.44){(\pr_2,\beta_{K,X}\cdot\pr_1)}
                    &&& \uS(Y,X^K)\times \uS(X^K,X)
                    \ar[d]^{\ub_{Y,X^K,X}}\\
                  K\times\usSet(K, \uS(Y,X))
                    \ar[rrr]_(0.58){\ev_{K,\usSet(K, \uS(Y,X))}}
                      &&& \uS(Y,X)}
              \end{equation*}
        \end{enumerate}
  \end{enumerate}
\end{definition}

Compared to the definitions of the cylinder object and the path object
in the closed model category,
the above definitions are more complicated.
So one may expect that they have additional properties
which are not apparent from the definitions.
In fact, there are functorial properties hidden.
Quillen use them in \cite{quillen-67}.
But he does not prove nor state it explicitly.
In \cite{hirschhorn-03},
the functorial properties are a part of the definition.

The aim of this note is to provide the proof.
To be more precise, we will prove that cylinder objects and path objects
in Definition \ref{def:esht.1.2} are unique up to
unique isomorphisms subject to the compatibility with morphisms in (b),
that objects in (a) are parts of functors and that
maps in (c) constitute natural isomorphisms between appropriate functors.
In particular, if they exist, they satisfy the axiom~\textbf{M6}
in \cite{hirschhorn-03}

Our first result is the uniqueness.

\begin{theorem}
\label{thm:esht.1.1}
  Let $\S$ be a simplicial category. Let $X\in\S$ and $K\in\sSet$.
  \begin{enumerate}
    \item
        Let
        $(X\otimes K,\alpha_{K,X},\{\phi_{K,X,Y}\}_{Y\in\S})$
        and $(X\otimes' K,\alpha_{K,X}',\{\phi_{K,X,Y}'\}_{Y\in\S})$
        be two $K$-cylinders of $X$.
        Then
        there exists a unique isomorphism
        \[f:X\otimes K\rightarrow X\otimes' K\] of $\S$
        such that
        \[
          \alpha_{K,X}'=\uS(X,f)\bullet\alpha_{K,X}.
        \]
        In particular, $X\otimes K$ and $X\otimes' K$ are isomorphic.
    \item
        Let
        $(X^K,\beta_{K,X},\{\psi_{K,X,Y}\}_{Y\in\S})$
        and
        $({X^K}',\beta_{K,X}',\{\psi_{K,X,Y}'\}_{Y\in\S})$
        be two $K$-paths of $X$.
        Then, there exists a unique isomorphism
        \[f:X^K\rightarrow {X^K}'\]
        such that
        \[
          \beta_{K,X}=\uS(f,X)\bullet\beta_{K,X}'.
        \]
        In particular, $X^K$ and $X^{K'}$ are isomorphic.
  \end{enumerate}
\end{theorem}

The second result is the functoriality.

\begin{theorem}
\label{thm:esht.1.2}
  Let $\S$ be a simplicial category.
  We assume that, for every $X\in\S$ and $K\in\sSet$,
  a $K$-cylinder of $X$ and a $K$-path of $X$ exist.
  \begin{enumerate}
    \item
        For each $(X,K)\in\S\times\sSet$,
        we fix a $K$-cylinder
        $(X\otimes K,\alpha_{K,X},\{\phi_{K,X,Y}\}_{Y\in\S})$ of $X$.
        Then, there exists a unique functor
        \[(-)\otimes(-):\S\times\sSet\rightarrow \S\]
        satisfying the following two conditions.
        \begin{enumerate}
          \item On objects, it maps $(X,K)$ to $X\otimes K$.
          \item
              If $L\in\sSet$,
              $Y\in\S$, $u\in\sSet(K,L)$ and $f\in\S(X,Y)$, then
              \[
                \xymatrix{
                  K \ar[r]^(0.3){\alpha_{K,X}} \ar[d]_u
                    & \uS(X,X\otimes K) \ar[d]^{\uS(X,X\otimes u)}\\
                  L \ar[r]_(0.3){\alpha_{L,X}} & \uS(X,X\otimes L)}
              \]
              and
              \[
                \xymatrix{
                  K \ar[rr]^(0.4){\alpha_{K,X}}
                    \ar[d]_{\alpha_{K,Y}}
                    && \uS(X,X\otimes K)
                    \ar[d]^{\uS(X,f\otimes K)}\\
                  \uS(Y,Y\otimes K) \ar[rr]_{\uS(f,Y\otimes K)}
                    && \uS(X, Y\otimes K)}
              \]
              commute.
        \end{enumerate}
    \item
        For each $(X,K)\in\S\times\sSet$,
        we fix a $K$-path
        $(X^K,\beta_{K,X},\{\psi_{K,X,Y}\}_{Y\in\S})$ of $X$.
        Then, there is a  unique functor
        \[(-)^{(-)}:\S\times\sSet\rightarrow \S\]
        satisfying the following two conditins.
        \begin{enumerate}
          \item
               On objects, it maps $(X,K)$ to $X^K$.
          \item
              If $L\in\sSet$,
              $Y\in\S$, $u\in\sSet(K,L)$ and $f\in\S(Y,X)$, then
              \begin{equation*}
                \xymatrix{
                  K \ar[r]^(0.3){\beta_{K,X}} \ar[d]_u
                    & \uS(X^K,X) \ar[d]^{\uS(X^u,X)}\\
                  L \ar[r]_(0.3){\beta_{L,X}} & \uS(X^L,X)}
              \end{equation*}
              and
              \begin{equation*}
                \xymatrix{
                  K \ar[rr]^(0.4){\beta_{K,X}} \ar[d]_{\beta_{K,Y}}
                    && \uS(X^K,X) \ar[d]^{\uS(f^K,X)}\\
                  \uS(Y^K,Y) \ar[rr]_{\uS(Y^K,f)}
                    && \uS(Y^K,X)}
              \end{equation*}
              commute.
        \end{enumerate}
  \end{enumerate}
\end{theorem}

Our final result is the naturality.

\begin{theorem}
\label{thm:esht.1.3}
  Let $\S$ be a simplicial category.
  We assume that, for every $X\in\S$ and $K\in\sSet$,
  a $K$-cylinder of $X$ and a $K$-path of $X$ exist.
  \begin{enumerate}
    \item
        For each $(X,K)\in\S\times\sSet$,
        we fix a $K$-cylinder
        $(X\otimes K,\alpha_{K,X},\{\phi_{K,X,Y}\}_{Y\in\S})$ of $X$.
        Then, $\phi_{K,X,Y}$ is natural in $K$, $X$ and $Y$.
    \item
        For each $(X,K)\in\S\times\sSet$,
        we fix a $K$-path
        $(X^K,\beta_{K,X},\{\psi_{K,X,Y}\}_{Y\in\S})$ of $X$.
        Then, $\psi_{K,Y,X}$ is natural in $K$, $X$ and $Y$.
    \item
        There are three adjunctions:
        \[
          \xymatrix{
            \sSet \ar@<0.8ex>[rr]^{X\otimes (-)}
              &&  \S \ar@<0.6ex>[ll]^{\uS(X, -)},}
          \xymatrix{
            \sSet \ar@<0.8ex>[rr]^{X^{(-)}}
              && \S^{\op} \ar@<0.6ex>[ll]^{\uS(-,X)}}\textrm{ and }
          \xymatrix{
            \S \ar@<0.8ex>[rr]^{(-)\otimes K}
              && \S \ar@<0.6ex>[ll]^{(-)^K}}
        \]
  \end{enumerate}
\end{theorem}

\begin{remark}
  Every result in this note is valid
  when $\sSet$ is replaced by the category of the finite simplicial sets.
\end{remark}

\begin{xnotations}
  \text{}
  \begin{itemize}
    \item we use $\sSet$ for the category of simplicial sets
    \item we use $\sigma^n_0:[n]\rightarrow[0]$ for the unique map from $[n]$ to $[0]$ in the category $\Delta$.
    \item we use $\bullet$ for the composition of arrows in simplicial categories.
  \end{itemize}
\end{xnotations}

\section{Review on Simplicial Category}

In this section, we recall the simplicial categories
and state some of their basic properties without proofs.
For references, we refer to \cite{gabriel-zisman-67},
\cite{quillen-67} and \cite{goerss-jardine-99}.

\subsection{The Axioms}

A simplicial category is a category enriched in the category
$\sSet$ of simplicial sets.
Here, we use the following equivalent definition in \cite{quillen-67},
which is convenient for us to prove the results
stated in section one.

\begin{definition}
\label{def:esht.1.1}
  A simplicial category is a  category $\S$ with the following structure
  \textbf{S1}-\textbf{S3} satisfying \textbf{S4} and
  \textbf{S5}.
  \begin{description}
    \item[S1] Functor
        $\uS(-,-):\S^{\op}\times\S \rightarrow \sSet$.
    \item[S2] For every $X,Y,Z\in\S$, there is a morphism
        \begin{equation*}
          \ub_{X,Y,Z}:\uS(X,Y)\times\uS(Y,Z) \rightarrow \uS(X,Z).
        \end{equation*}
        of $\sSet$.
        (For simplicity, we denote $\ub_{X,Y,Z}$ by $\ub$.)
    \item[S3]
        A natural isomorphism
        \begin{equation*}
          \widetilde{(\;\;)}:\S(-,-) \rightarrow \uS(-,-)_0
        \end{equation*}
        between two functors $\S(-,-)$ and $\uS(-,-)_0$
        from $\S^{\op}\times\S$ to $\Set$.
    \item[S4]
        If $f\in\uS(X,Y)_n$,
        $g\in\uS(Y,Z)_n$, $h\in\uS(Z,W)_n$, then
        \begin{equation*}
          (h\ub_n g)\ub_n f=h\ub_n(g\ub_n f)
        \end{equation*}
        holds.
    \item[S5]
        \begin{enumerate}
          \item If $f\in\uS(X,Y)_n$,
              $g\in\S(Y,Z)$, then
              \begin{equation*}
                \uS(X,g)_n(f)=(\sigma_0^n)^*(\widetilde{g})\ub_n f
              \end{equation*}
              holds.
          \item If $f\in\S(X,Y)$,
              $g\in\uS(Y,Z)_n$, then
              \begin{equation*}
                \uS(f,Z)_n(g)= g\ub_n (\sigma_0^n)^*(\widetilde{f})
              \end{equation*}
              holds.
        \end{enumerate}
  \end{description}
\end{definition}

\begin{remark}
  Let $\S$ be a simplicial category.
  If $X,Y,Z\in\S$, then
  \begin{equation*}
    \xymatrix{
      \uS(X,Y)_0\times \uS(Y,Z)_0 \ar[r]^(0.6){\ub_0} & \uS(X,Z)_0\\
      \S(X,Y)\times \S(Y,Z) \ar[r]^(0.6){\bullet}
        \ar[u]^{\widetilde{(\;\;)}\times\widetilde{(\;\;)}}
        & \S(X,Z) \ar[u]_{\widetilde{(\;\;)}}}
  \end{equation*}
  is a commutative diagram.
\end{remark}

\begin{definition}
  Let $\S$ be a simplicial category.
  Let $f\in\uS(X,Y)_n$ and $g\in\uS(Y,Z)_n$.
  \begin{enumerate}
    \item
        We indicate $f\in\uS(X,Y)_n$ by
        \[
          \xymatrix{f:X\ar[r]|(0.55)n &Y}\textrm{, or }
          \xymatrix{X\ar[r]|n^f &Y}.
        \]
    \item
        We indicate $g\ub_n f\in\uS(X,Z)_n$ by
        \[
          \xymatrix{X\ar[r]|n^f & Y \ar[r]|n^g & Z}.
        \]
  \end{enumerate}
\end{definition}

\subsection{$\sSet$ as a Simplicial Category}
\label{subsubsec:esht.1.2.1}

\begin{definition}[\textbf{S1}]
\label{def:esht.1.9}
  Let $X,Y,Z\in\sSet$.
  \begin{enumerate}
    \item We define an object $\usSet(X,Y)$ of $\sSet$ by
        \begin{equation*}
          \usSet(X,Y)_n:=\sSet(X\times\Delta[n],Y).
        \end{equation*}
        and
        \begin{equation*}
          \usSet(X,Y)(\theta):=(\id_X\times\Delta[\theta])^*
        \end{equation*}
        where $\theta\in\Delta(m,n)$.
    \item
        Let $u\in\sSet(Z,X)$. We define a morphism
        \[
          \usSet(u,Y):\usSet(X,Y)\rightarrow\usSet(Z,Y)
        \]
        of $\sSet$ by
        \begin{equation*}
          \usSet(u,Y)_n(f):=f\bullet(u\times\id_{\Delta[n]})
        \end{equation*}
        where $f\in\usSet(X,Y)_n$.
    \item
        Let $u\in\sSet(Y,Z)$.
        We define a morphism
        \[
          \usSet(X,u):\usSet(X,Y)\rightarrow\usSet(X,Z)
        \]
        of $\sSet$ by
        \[
          \usSet(X,u)_n(f):=u\bullet f
        \]
        where $f\in\usSet(X,Y)_n$.
  \end{enumerate}
\end{definition}

\begin{definition}[\textbf{S2}]
\label{def:esht.1.10}
  Let $X,Y,Z\in\sSet$.
  We define a morphism
  \begin{equation*}
    \ub_{X,Y,Z}:\usSet(X,Y)\times\usSet(Y,Z)\rightarrow \usSet(X,Z)
  \end{equation*}
  of $\sSet$ by
  \begin{equation*}
    (\ub_{X,Y,Z})_n(f,g):=g\bullet(f\times\id_{\Delta[n]})
      \bullet(\id_X\times\delta_{\Delta[n]}).
  \end{equation*}
  where \[\delta_{\Delta[n]}:\Delta[n]\rightarrow\Delta[n]\times \Delta[n]\]
  is the diagonal map.
\end{definition}

\begin{definition}[\textbf{S3}]
\label{def:esht.1.11}
  Let $X,Y\in\sSet$.
  We define a morphism
  \begin{equation*}
    \widetilde{(\;\;)}:\sSet(X,Y)\rightarrow \usSet(X,Y)_0
  \end{equation*}
  of $\Set$ by
  \begin{equation*}
    \widetilde{f}:=f\bullet r_X
  \end{equation*}
  where
  \[
    r_X:X\times\Delta[0]\rightarrow X
  \]
  is the canonical isomorphism.
\end{definition}

\begin{lemma}
  $\sSet$  with the structures in
  Definition \ref{def:esht.1.9} - Definition \ref{def:esht.1.11}
  is a simplicial category.
\end{lemma}

\subsection{$\ev_{X,Y}$ and $\sharp$}

\begin{definition}
\label{def:esht.1.14}
  Let $X,Y\in\sSet$. We define a morphism
  \begin{equation*}
    \ev_{X,Y}:X\times\usSet(X,Y)\rightarrow Y
  \end{equation*}
  of $\sSet$ by
  \begin{equation*}
    (\ev_{X,Y})_n(x_n,f):=f_n(x_n,\id_n).
  \end{equation*}
\end{definition}

\begin{lemma}
\label{lem:esht.1.15}
  Let $X,Y,Z\in\sSet$. If $f\in\sSet(Y,Z)$, then
  \[
    \xymatrix{
      X\times \usSet(X,Y) \ar[rr]^(0.65){\ev_{X,Y}}
        \ar[d]_{\id_X\times \usSet(X,f)}
        && Y \ar[d]^{f}\\
      X\times \usSet(X,Z) \ar[rr]_(0.65){\ev_{X,Z}} && Z
      }
  \]
  is a commutative diagram.
\end{lemma}

\begin{lemma}
\label{lem:esht.1.16}
  Let $X,Y,Z\in\sSet$. If $f\in\sSet(X,Y)$, then
  \[
    \xymatrix{
      X\times \usSet(Y,Z) \ar[rr]^{f\times \id_{\usSet(Y,Z)}}
        \ar[d]_{\id_X\times \usSet(f,Z)}
        && Y\times \usSet(Y,Z) \ar[d]^{\ev_{Y,Z}}\\
      X\times \usSet(X,Z) \ar[rr]_{\ev_{X,Z}} && Z
      }
  \]
  is commutative diagram.
\end{lemma}

\begin{definition}
\label{def:esht.1.16}
  Let $K,X,Y\in\sSet$.
  We define a morphism
  \begin{equation*}
    \sharp_{K,X,Y}:\sSet(K,\usSet(X,Y))\rightarrow \sSet(X\times K,Y)
  \end{equation*}
  of $\sSet$ by
  \begin{equation*}
    \sharp_{K,X,Y}(u):=\ev_{X,Y}\bullet(\id_X\times u).
  \end{equation*}
  I.e.,
  \begin{equation*}
    \xymatrix{
      X\times K
        \ar[d]_{\id_X\times u} \ar[drrr]|{\sharp_{K,X,Y}(u)}
        &&& \\
      X\times\usSet(X,Y) \ar[rrr]_(0.58){\ev_{X,Y}} &&& Y}
  \end{equation*}
  is a commutative diagram and
  \begin{equation*}
    (\sharp_{K,X,Y}(u))_n(x_n,k_n)=(u_n(k_n))_n(x_n,\id_n)
  \end{equation*}
  holds.
\end{definition}

\begin{lemma}
\label{lem:esht.1.18}
  For every $K,X,Y\in\sSet$,
  $\sharp_{K,X,Y}$ in Definition \ref{def:esht.1.16}
  is an isomorphism and
  natural in $K$, $X$ and $Y$.
\end{lemma}

\begin{remark}
  In this note, we only use the fact that
  $\sharp_{K,X,Y}$ is an isomorphism.
\end{remark}

\section{Proofs}

\begin{lemma}
\label{lem:esht.1.24}
  Let $\S$ be a simplicial category.
  Let $K\in\sSet$. Let $X,Y,Z\in\S$.
  If two morphisms
  \[
    \alpha:K\rightarrow\uS(X,Z)
  \]
  and
  \[
    \phi:\uS(Z,Y)\rightarrow \usSet(K,\uS(X,Y))
  \]
  of $\sSet$ satisfy
  \[
    \sharp_{\uS(Z,Y),K,\uS(X,Y)}(\phi)
    =(\ub_{X,Z ,Y})\bullet(\alpha\times\id_{\uS(Z,Y)})
  \]
  then, for every $f\in\S(Z,Y)$,
  \[
    (\phi_0)(\widetilde{f})=\widetilde{\uS(X,f)\bullet\alpha}
  \]
  holds.
\end{lemma}
\begin{proof}
  Let $k_n\in K_n$.
  Then
  \begin{alignat*}{1}
    (\sharp_{\uS(Z,Y),K,\uS(X,Y)}(\phi))_n
      (k_n,(\sigma^n_0)^*(\widetilde{f}))
    =& (\phi_n((\sigma^n_0)^*(\widetilde{f})))_n(k_n,\id_n)\\
    =&((\sigma^n_0)^*((\phi_0)(\widetilde{f})))_n(k_n,\id_n)\\
    =&((\phi_0)(\widetilde{f}))_n(k_n,\sigma^n_0)
  \end{alignat*}
  and
  \begin{alignat*}{1}
    ((\ub_{X,Z ,Y})\bullet(\alpha\times\id_{\uS(Z,Y)}))_n
      (k_n,(\sigma^n_0)^*(\widetilde{f}))
    =& (\sigma^n_0)^*(\widetilde{f})\ub_{X,Z ,Y,n}\alpha_n(k_n)\\
    =&\uS(X,f)_n(\alpha_n(k_n))\\
    =&(\uS(X,f)\bullet\alpha)_n(k_n)\\
    =&\widetilde{(\uS(X,f)\bullet\alpha)}_n(k_n,\sigma^n_0)
  \end{alignat*}
  holds.
\end{proof}

The following lemma is an immediate consequence of
Lemma \ref{lem:esht.1.24}.

\begin{lemma}
\label{lem:esht.1.25}
  Let $\S$ be a simplicial category.
  Let $K,L\in\sSet$. Let $X\in\S$.
  We assume that a $K$-cylinder  and a $L$-cylinder of $X$ exists.
  If $u\in\sSet(K,L)$ and $f\in\S(X\otimes K,X\otimes L)$,
  then the following conditions are equivalent.
  \begin{enumerate}
    \item $(\phi_{K,X,X\otimes L})_0(\widetilde{f})
        =\widetilde{\alpha_{L,X}\bullet u}$.
    \item $\uS(X,f)\bullet\alpha_{K,X}=\alpha_{L,X}\bullet u$.
        I.e.,
        \begin{equation*}
          \xymatrix{
            K \ar[r]^(0.3){\alpha_{K,X}} \ar[d]_u
              & \uS(X,X\otimes K) \ar[d]^{\uS(X,f)}\\
            L \ar[r]_(0.3){\alpha_{L,X}} & \uS(X,X\otimes L)}
        \end{equation*}
        is a commutative diagram.
  \end{enumerate}
\end{lemma}

\begin{definition}
  Let $\S$ be a simplicial category.
  Let $K,L\in\sSet$. Let $X\in\S$.
  We assume that a $K$-cylinder  and a $L$-cylinder of $X$ exists.
  Since $\phi_{K,X,X\otimes L}$ is an isomorphism,
  given any $u\in\sSet(K,L)$
  there exists a unique morphism $f\in\S(X\otimes K,X\otimes L)$
  satisfying the equivalent conditions of
  Lemma \ref{lem:esht.1.25}.
  We denote this morphism with
  \[
    X\otimes u:X\otimes K\rightarrow X\otimes L.
  \]
  In other words, we have the following commutative diagram.
  \begin{equation}
  \label{eqn:esht.1.1}
    \xymatrix{
      K \ar[r]^(0.3){\alpha_{K,X}} \ar[d]_u
        & \uS(X,X\otimes K) \ar[d]^{\uS(X,X\otimes u)}\\
      L \ar[r]_(0.3){\alpha_{L,X}} & \uS(X,X\otimes L)}
  \end{equation}
\end{definition}

The following lemma is also an immediate consequence of
Lemma \ref{lem:esht.1.24}.

\begin{lemma}
\label{lem:esht.1.27}
  Let $\S$ be a simplicial category.
  Let $K\in\sSet$. Let $X,Y\in\S$.
  We assume that a $K$-cylinder of $X$ and a $K$-cylinder of $Y$ exist.
  If $u\in\S(X,Y)$ and $f\in\S(X\otimes K,Y\otimes K)$, then
  the following conditions are equivalent.
  \begin{enumerate}
    \item $(\phi_{K,X,Y\otimes K})_0(\widetilde{f})
        =\widetilde{\uS(u,Y\otimes K)\bullet \alpha_{K,Y}}$.
    \item $\uS(X,f)\bullet\alpha_{K,X}=
        \uS(u,Y\otimes K)\bullet \alpha_{K,Y}$. I.e.,
        \begin{equation*}
          \xymatrix{
            K \ar[rr]^(0.4){\alpha_{K,X}} \ar[d]_{\alpha_{K,Y}}
              && \uS(X,X\otimes K) \ar[d]^{\uS(X,f)}\\
            \uS(Y,Y\otimes K) \ar[rr]_{\uS(u,Y\otimes K)}
              && \uS(X, Y\otimes K)}
        \end{equation*}
        is a commutative diagram.
  \end{enumerate}
\end{lemma}

\begin{definition}
  Let $\S$ be a simplicial category.
  Let $K\in\sSet$. Let $X,Y\in\S$.
  We assume that a $K$-cylinder of $X$ and a $K$-cylinder of $Y$ exist.
  Since
  $\phi_{K,X,Y\otimes K}$ is an isomorphism,
  given any $u\in\S(X,Y)$
  there exists a unique morphism $f\in\S(X\otimes K,Y\otimes K)$
  satisfying the equivalent conditions of Lemma \ref{lem:esht.1.27}.
  We denote this morphism with
  \[
    u\otimes K:X\otimes K\rightarrow Y\otimes K.
  \]
  In other words, we have the following commutative diagram.
  \begin{equation}
  \label{eqn:esht.1.2}
    \xymatrix{
      K \ar[rr]^(0.4){\alpha_{K,X}} \ar[d]_{\alpha_{K,Y}}
        && \uS(X,X\otimes K) \ar[d]^{\uS(X,u\otimes K)}\\
      \uS(Y,Y\otimes K) \ar[rr]_{\uS(u,Y\otimes K)}
        && \uS(X, Y\otimes K)}
  \end{equation}
\end{definition}

\begin{lemma}
\label{lem:esht.1.30}
  Let $\S$ be a simplicial category.
  Let $K\in\sSet$. Let $X,Y\in\S$.
  We assume that a $K$-cylinder of $X$ and a $K$-cylinder of $Y$ exist.
  If $u\in\S(X,Y)$ and $k_n\in K_n$,
  \[
    (\sigma^n_0)^*(\widetilde{u\otimes K})\ub_n(\alpha_{K,X})_n(k_n)
    =(\alpha_{K,Y})_n(k_n)\ub_n(\sigma^n_0)^*(\widetilde{u})
  \]
  holds. I.e.,
  \begin{equation*}
    \xymatrix{
      X \ar[rr]^(0.4){(\alpha_{K,X})_n(k_n)}|n
        \ar[d]_{(\sigma^n_0)^*(\widetilde{u})\;}|n
        && X\otimes K \ar[d]^{\;(\sigma^n_0)^*(\widetilde{u\otimes K})}|n\\
      Y \ar[rr]_(0.4){(\alpha_{K,Y})_n(k_n)}|n
        && Y\otimes K}
  \end{equation*}
  is a commutative diagram.
\end{lemma}
\begin{proof}
  It follows from the commutative diagram \eqref{eqn:esht.1.2}.
\end{proof}

\begin{lemma}
\label{lem:esht.1.31}
  Let $\S$ be a simplicial category.
  Let $K\in\sSet$. Let $X,Y,Z\in\S$.
  We assume that $K$-path of $X$ exists.
  If two morphism
  \[
    \beta:K\rightarrow\uS(Z,X)
  \]
  and
  \[
    \psi:\uS(Y,Z)\rightarrow \usSet(K,\uS(Y,X))
  \]
  of $\sSet$ satisfy
  \[
    \sharp_{\uS(Y,Z),K,\uS(Y,X)}(\psi)=
      (\ub_{Y,Z,X})\bullet (\pr_2,\beta\cdot\pr_1)
  \]
  then, for every $f\in\S(Y,Z)$,
  \[
    (\psi_0)(\widetilde{f})=\widetilde{\uS(f,X)\bullet\beta}
  \]
  holds.
\end{lemma}
\begin{proof}
  Let $k_n\in K_n$. Then
  \begin{alignat*}{1}
    (\sharp_{\uS(Y,Z),K,\uS(Y,X)}(\psi))_n
      (k_n,(\sigma^n_0)^*(\widetilde{f}))
    =& (\psi_n((\sigma^n_0)^*(\widetilde{f})))_n(k_n,\id_n)\\
    =&((\sigma^n_0)^*((\psi_0)(\widetilde{f})))_n(k_n,\id_n)\\
    =&((\psi_0)(\widetilde{f}))_n(k_n,\sigma^n_0)
  \end{alignat*}
  and
  \begin{alignat*}{1}
    ((\ub_{Y,Z,X})\bullet (\pr_2,\beta\cdot\pr_1))_n
      (k_n,(\sigma^n_0)^*(\widetilde{f}))
    =&\beta_n(k_n)\ub_n(\sigma^n_0)^*(\widetilde{f})\\
    =&\uS(f,X)_n(\beta_n(k_n))\\
    =&(\uS(f,X)\bullet\beta)_n(k_n)\\
    =&\widetilde{(\uS(f,X)\bullet\beta)}_n(k_n,\sigma^n_0)
  \end{alignat*}
  holds.
\end{proof}

The following lemma is an immediate consequence of
Lemma \ref{lem:esht.1.31}.

\begin{lemma}
\label{lem:esht.1.32}
  Let $\S$ be a simplicial category.
  Let $K,L\in\sSet$. Let $X\in\S$.
  We assume that a $K$-path of $X$ and a $L$-path of $X$ exist.
  If $u\in\sSet(K,L)$ and $f\in\S(X^L,X^K)$,
  then the following conditions are equivalent.
  \begin{enumerate}
    \item $(\psi_{K,X^L,X})_0(\widetilde{f})
        =\widetilde{\beta_{L,X}\bullet u}$.
    \item $\uS(f,X)\bullet\beta_{K,X}=\beta_{L,X}\bullet u$. I.e.,
        \begin{equation*}
          \xymatrix{
            K \ar[r]^(0.3){\beta_{K,X}} \ar[d]_u
              & \uS(X^K,X) \ar[d]^{\uS(f,X)}\\
            L \ar[r]_(0.3){\beta_{L,X}} & \uS(X^L,X)}
        \end{equation*}
        은 commutative diagram이다.
  \end{enumerate}
\end{lemma}

\begin{definition}
  Let $\S$ be a simplicial category.
  Let $K,L\in\sSet$. Let $X\in\S$.
  We assume that a $K$-path of $X$ and a $L$-path of $X$ exist.
  Since $\psi_{K,X^L,X}$ is an isomorphism,
  given any $u\in\sSet(K,L)$, there exists a unique morphism  $f\in\S(X^L,X^K)$ satisfying the equivalent condition of
  Lemma \ref{lem:esht.1.32}.
  We denote this morphism by
  \[
    X^u:X^L\rightarrow X^K.
  \]
  In other words, we have the following commutative diagram.
  \begin{equation*}
    \xymatrix{
      K \ar[r]^(0.3){\beta_{K,X}} \ar[d]_u
        & \uS(X^K,X) \ar[d]^{\uS(X^u,X)}\\
      L \ar[r]_(0.3){\beta_{L,X}} & \uS(X^L,X)}
  \end{equation*}
\end{definition}

The following lemma is also an immediate consequence of
Lemma \ref{lem:esht.1.31}.

\begin{lemma}
\label{lem:esht.1.34}
  Let $\S$ be a simplicial category.
  Let $K\in\sSet$. Let $X,Y\in\S$.
  We assume that a $K$-path of $X$ and a $K$-path of $Y$ exist.
  If $u\in\S(Y,X)$ and $f\in\S(Y^K,X^K)$,
  then the following conditions are equivalent.
  \begin{enumerate}
    \item $(\psi_{K,Y^K,X})_0(\widetilde{f})
        =\widetilde{\uS(Y^K,u)\bullet \beta_{K,Y}}$.
    \item $\uS(f,X)\bullet\beta_{K,X}=
        \uS(Y^K,u)\bullet \beta_{K,Y}$. I.e.,
        \begin{equation*}
          \xymatrix{
            K \ar[rr]^(0.4){\beta_{K,X}} \ar[d]_{\beta_{K,Y}}
              && \uS(X^K,X) \ar[d]^{\uS(f,X)}\\
            \uS(Y^K,Y) \ar[rr]_{\uS(Y^K,u)}
              && \uS(Y^K,X)}
        \end{equation*}
        is a commutative diagram.
  \end{enumerate}
\end{lemma}

\begin{definition}
  Let $\S$ be a simplicial category.
  Let $K\in\sSet$. Let $X,Y\in\S$.
  We assume that a $K$-path of $X$ and a $K$-path of $Y$ exist.
  Since $\psi_{K,Y^K,X}$ is an isomorphism,
  given any $u\in\S(Y,X)$,
  there exists a unique morphism $f\in\S(Y^K,X^K)$
  satisfying the equivalent condition of
  Lemma \ref{lem:esht.1.34}.
  We denote this morphism with
  \[
    u^K:Y^K\rightarrow X^K.
  \]
  In other words, we have the following commutative diagram.
  \begin{equation}
  \label{eqn:esht.1.3}
    \xymatrix{
      K \ar[rr]^(0.4){\beta_{K,X}} \ar[d]_{\beta_{K,Y}}
        && \uS(X^K,X) \ar[d]^{\uS(u^K,X)}\\
      \uS(Y^K,Y) \ar[rr]_{\uS(Y^K,u)}
        && \uS(Y^K,X)}
  \end{equation}
\end{definition}

\begin{lemma}
\label{lem:esht.1.36}
  Let $\S$ be a simplicial category.
  Let $K\in\sSet$. Let $X,Y\in\S$.
  We assume that a $K$-path of $X$ and a $K$-path of $Y$ exist.
  If $u\in\S(Y,X)$ and $k_n\in K_n$, then
  \[
    (\sigma^n_0)^*(\widetilde{u})\ub_n(\beta_{K,Y})_n(k_n)=
      (\beta_{K,X})_n(k_n)\ub_n(\sigma^n_0)^*(\widetilde{u^K})
  \]
  holds.
  I.e.,
  \[
    \xymatrix{
      Y^K \ar[d]|n_{(\sigma^n_0)^*(\widetilde{u^K})\;}
        \ar[rr]|n^{(\beta_{K,Y})_n(k_n)}
        && Y \ar[d]|n^{\;(\sigma^n_0)^*(\widetilde{u})}\\
      X^K \ar[rr]|n_{(\beta_{K,X})_n(k_n)} && X}
  \]
  is a commutative diagram.
\end{lemma}
\begin{proof}
  It follows from the commutative diagram \eqref{eqn:esht.1.3}.
\end{proof}

Now we give the proof of Theorem \ref{thm:esht.1.1}.

\begin{proof}[Proof of Theorem \ref{thm:esht.1.1}]

(1)
  Since $\phi_{K,X,X\otimes'K}$ is an isomorphism,
  by Lemma \ref{lem:esht.1.24},
  there exists a unique morphism $f\in\S(X\otimes K,X\otimes'K)$
  making
  \begin{equation*}
    \xymatrix{
      K \ar[r]^(0.3){\alpha_{K,X}} \ar[d]_{\id_K}
        & \uS(X,X\otimes K) \ar[d]^{\uS(X,f)}\\
      K \ar[r]_(0.3){\alpha_{K,X}'} & \uS(X,X\otimes'K)}
  \end{equation*}
  commute.

(2)
  Since $\psi_{K,X,X^{K'}}$ is an isomorphism,
  by Lemma \ref{lem:esht.1.31},
  there exists a unique morphism $f\in\S(X^{K'},X^{K})$
  making
  \begin{equation*}
    \xymatrix{
      K \ar[r]^(0.3){\beta_{K,X}} \ar[d]_{\id_K}
        & \uS(X^K,X) \ar[d]^{\uS(f,X)}\\
      K \ar[r]_(0.3){\beta_{K,X}'} & \uS(X^{K'},X)}
  \end{equation*}
  commute.
\end{proof}

Next, we give the proof of Theorem \ref{thm:esht.1.2}.

\begin{proof}[Proof of Theorem \ref{thm:esht.1.2}]
(1) Let $X,Y,Z\in\S$. Let $K,L,M\in\sSet$.
  For every $u:K\rightarrow L$ and $f:X\rightarrow Y$,
  we have already defined $X\otimes u$ and $f\otimes K$.

(a) $X\otimes (-)$ is a functor:
  Let $u\in\S(K,L)$ and $v\in\S(L,M)$.
  Then by the uniqueness of $X\otimes (v\bullet u)$,
  \[
    (X\otimes v)\bullet (X\otimes u)=X\otimes (v\bullet u)
  \]
  holds. Similarly, by the uniqueness of $X\otimes \id_K$,
  \[
    \id_{X\otimes K}=X\otimes \id_K
  \]
  holds.

(b) $(-)\otimes K$ is a functor:

(i)
  Let $u\in\S(X,Y)$ and $v\in\S(Y,Z)$.
  Since $\uS(-,-)$ is a functor,
  the parallelogram at the bottom of the following diagram
  is a commutative diagram.
  \begin{equation*}
    \xymatrix{
      K \ar[d]_= \ar[r]^(0.3){\alpha_{K,X}}
        & \uS(X,X\otimes K) \ar[drr]^{\;\;\uS(X,u\otimes K)}
        & & \\
      K \ar[d]_= \ar[r]^(0.3){\alpha_{K,Y}}
        & \uS(Y,Y\otimes K) \ar[rr]^{\uS(u,Y\otimes K)}
          \ar[drr]^{\;\;\uS(Y,v\otimes K)}
        && \uS(X,Y\otimes K) \ar[drr]^{\;\;\uS(X,v\otimes K)} && \\
      K \ar[r]_(0.3){\alpha_{K,Z}}
        & \uS(Z,Z\otimes K) \ar[rr]_{\uS(v,Z\otimes K)}
        && \uS(Y,Z\otimes K) \ar[rr]_{\uS(u,Z\otimes K)}
        && \uS(X,Z\otimes K)}
  \end{equation*}
  By Lemma \ref{lem:esht.1.27}, the rest also commute.
  Therefore,
  \begin{equation*}
    \xymatrix{
      K \ar[rr]^(0.4){\alpha_{K,X}} \ar[d]_{\alpha_{K,Z}}
        && \uS(X,X\otimes K) \ar[d]^{\uS(X,(v\otimes K)\bullet(u\otimes K))}\\
      \uS(Z,Z\otimes K) \ar[rr]_{\uS(v\bullet u,Z\otimes K)}
        && \uS(X, Z\otimes K)}
  \end{equation*}
  is a commutative diagram.
  Then, by the uniqueness of $(v\bullet u)\otimes K$,
  \[
    (v\otimes K)\bullet(u\otimes K)=(v\bullet u)\otimes K
  \]
  holds.

(ii) Since $\uS(-,-)$ is a functor,
  \begin{equation*}
    \xymatrix{
      K \ar[rr]^(0.4){\alpha_{K,X}} \ar[d]_{\alpha_{K,X}}
        && \uS(X,X\otimes K) \ar[d]^{\uS(X,\id_{X\otimes K})}\\
      \uS(X,X\otimes K) \ar[rr]_{\uS(\id_X,X\otimes K)}
        && \uS(X, X\otimes K)}
  \end{equation*}
  is also commutative diagram.
  Then, again by the uniqueness of $\id_X\otimes K$,
  \[
    \id_{X\otimes K}=\id_X\otimes K
  \]
  holds.

(c) $(-)\otimes (-)$ is a bifunctor:
  Let $u\in\sSet(K,L)$ and $f\in\S(X,Y)$.
  $\uS(-,-)$는 functor이다 .
  따라서 다음 diagram의 앞면은 commutative diagram이다.
  Consider the following diagram.
  \[
    \xymatrix{
      K \ar[rr]^{\alpha_{K,X}} \ar[dr]^{\alpha_{K,Y}} \ar[dd]_u
        & & \uS(X,X\otimes K) \ar[dr]^{\;\;\uS(X,f\otimes K)}
        \ar[dd]|\hole^(0.7){\uS(X,X\otimes u)}& \\
      & \uS(Y,Y\otimes K) \ar[rr]^(0.35){\uS(f,Y\otimes K)}
        \ar[dd]^(0.7){\uS(Y,Y\otimes u)}
        & & \uS(X,Y\otimes K) \ar[dd]^{\uS(X,Y\otimes u)}\\
      L \ar[rr]^(0.2){\alpha_{L,X}}|(0.42)\hole \ar[dr]_{\alpha_{L,Y}}
        & & \uS(X,X\otimes L) \ar[dr]^{\;\;\uS(X,f\otimes L)} & \\
      & \uS(Y,Y\otimes L) \ar[rr]_{\uS(f,Y\otimes L)}& & \uS(X,Y\otimes L)}
  \]
  Every face commutes except possibly the right one.
  So,
  \[
    \uS(X,(Y\otimes u)\bullet(f\otimes K))\bullet\alpha_{K,X}
    =
    \uS(X,(f\otimes L)\bullet(X\otimes u))\bullet\alpha_{K,X}.
  \]
  Then, by Lemma \ref{lem:esht.1.24}
  \[
    (\phi_{K,X,Y\otimes L})_0(\widetilde{(Y\otimes u)\bullet(f\otimes K)})
    =
    (\phi_{K,X,Y\otimes L})_0(\widetilde{(f\otimes L)\bullet(X\otimes u)}).
  \]
  Then, since $\phi_{K,X,Y\otimes L}$ is an isomorphism,
  \[
    (Y\otimes u)\bullet(f\otimes K)=(f\otimes L)\bullet(X\otimes u).
  \]

(2)

(a) $X^{(-)}$ is a functor:
  Let $u\in\S(K,L)$ and $v\in\S(L,M)$.
  By the uniqueness of $X^{(v\bullet u)}$
  \[
    (X^u)\bullet (X^v)=X^{(v\bullet u)}
  \]
  holds. Similarly, by the uniqueness of $X^{\id_K}$,
  \[
    \id_{X^K}=X^{\id_K}
  \]
  holds.

(b) $(-)^K$ is a functor다:

(i)
  Let $u\in\S(Y,X)$ and $v\in\S(Z,Y)$.
  Since $\uS(-,-)$ is a functor,
  the parallelogram at the bottom of the following diagram
  is a commutative diagram.
  \begin{equation*}
    \xymatrix{
      K \ar[d]_= \ar[r]^(0.3){\beta_{K,X}}
        & \uS(X^K,X) \ar[drr]^{\;\;\uS(u^K,X)}
        & & \\
      K \ar[d]_= \ar[r]^(0.3){\beta_{K,Y}}
        & \uS(Y^K,Y) \ar[rr]^{\uS(Y^K,u)}
          \ar[drr]^{\;\;\uS(v^K,Y)}
        && \uS(Y^K,X) \ar[drr]^{\;\;\uS(v^K,X)} && \\
      K \ar[r]_(0.3){\beta_{K,Z}}
        & \uS(Z^K,Z) \ar[rr]_{\uS(Z^K,v)}
        && \uS(Z^K,Y) \ar[rr]_{\uS(Z^K,u)} && \uS(Z^K,X)}
  \end{equation*}
  By Lemma \ref{lem:esht.1.34}, the rest also commute.
  Therefore
  \begin{equation*}
    \xymatrix{
      K \ar[rr]^(0.4){\beta_{K,X}} \ar[d]_{\beta_{K,Z}}
        && \uS(X^K,X) \ar[d]^{\uS((u^K)\bullet(v^K),X)}\\
      \uS(Z^K,Z) \ar[rr]_{\uS(Z^K,u\bullet v)}
        && \uS(Z^K,X)}
  \end{equation*}
  is a commutative diagram.
  Then, by the uniqueness of $(u\bullet v)^K$,
  \[
    (u^K)\bullet(v^K)=(u\bullet v)^K
  \]
  holds.

(ii) Since $\uS(-,-)$ is a  functor,
  \begin{equation*}
    \xymatrix{
      K \ar[rr]^(0.4){\beta_{K,X}} \ar[d]_{\beta_{K,X}}
        && \uS(X^K,X) \ar[d]^{\uS(\id_{X^K},X)}\\
      \uS(X^K,X) \ar[rr]_{\uS(X^K,\id_X)}
        && \uS(X^K,X)}
  \end{equation*}
  is a commutative diagram.
  Thus, by the uniqueness of $(\id_X)^K$,
  \[
    \id_{X^K}=(\id_X)^K
  \]
  holds.

(c) $(-)^{(-)}$ is a bifunctor다:
  Let $u\in\sSet(K,L)$ and $f\in\S(Y,X)$.
  $\uS(-,-)$는 functor이다.
  Consider the following commutative diagram.
  \[
    \xymatrix{
      K \ar[rr]^{\beta_{K,X}} \ar[dr]^{\beta_{K,Y}} \ar[dd]_u
        & & \uS(X^ K,X) \ar[dr]^{\;\;\uS(f^ K,X)}
        \ar[dd]|\hole^(0.7){\uS(X^u,X)}& \\
      & \uS(Y^ K,Y) \ar[rr]^(0.35){\uS(Y^ K,f)}
        \ar[dd]^(0.7){\uS(Y^ u,Y)}
        & & \uS(Y^ K,X) \ar[dd]^{\uS(Y^ u,X)}\\
      L \ar[rr]^(0.2){\beta_{L,X}}|(0.42)\hole \ar[dr]_{\beta_{L,Y}}
        & & \uS(X^ L,X) \ar[dr]^{\;\;\uS(f^ L,X)} & \\
      & \uS(Y^ L,Y) \ar[rr]_{\uS(Y^ L,f)}& & \uS(Y^ L,X)}
  \]
  Every face commutes except possibly the right one.
  So,
  \[
    \uS((f^ K)\bullet(Y^ u),X)\bullet\beta_{K,X}
    =
    \uS((X^ u)\bullet(f^ L),X)\bullet\beta_{K,X}.
  \]
  Then by  Lemma \ref{lem:esht.1.31}
  \[
    (\psi_{K,Y^L,X})_0(\widetilde{(f^ K)\bullet(Y^ u)})
    =
    (\psi_{K,Y^L,X})_0(\widetilde{(X^ u)\bullet(f^ L)}).
  \]
  Then, since $\psi_{K,Y^L,X}$ is an isomorphism,
  \[
    (f^ K)\bullet(Y^ u)=(X^ u)\bullet(f^ L).
  \]
\end{proof}

Finally, we prove Theorem \ref{thm:esht.1.3}.

\begin{proof}[Proof of Theorem \ref{thm:esht.1.3}]
(1) Let $K,L\in\sSet$. Let $W,X,Y,Z\in\S$.

(a)
  $\phi_{-,X,Y}$: Let $u\in\sSet(K,L)$.
  By Lemma \ref{lem:esht.1.18},
  we need to show that
  \begin{equation}
  \label{eqn:esht.1.4}
    \sharp(\phi_{K,X,Y}\bullet\uS(X\otimes u,Y))
    =\sharp(\usSet(u,\uS(X,Y))\bullet\phi_{L,X,Y})
  \end{equation}
  holds where $\sharp=\sharp_{\uS(X\otimes L,Y),K,\uS(X,Y)}$.
  \[
    \xymatrix{
      \uS(X\otimes K,Y) \ar[rr]^{\phi_{K,X,Y}}
        && \usSet(K,\uS(X,Y))\\
      \uS(X\otimes L,Y) \ar[rr]_{\phi_{L,X,Y}} \ar[u]^{\uS(X\otimes u,Y)} &
        & \usSet(L,\uS(X,Y)) \ar[u]_{\usSet(u,\uS(X,Y))}}
  \]

(i)
  By the definition of $\sharp$ and $\phi_{K,X,Y}$,
  the dotted arrows of the following diagram computes
  $\sharp(\phi_{K,X,Y}\bullet\uS(X\otimes u,Y))$.
  \[
    \xymatrix{
      K\times \uS(X\otimes L, Y)
        \ar@{.>}[d]_{\id_K\times\uS(X\otimes u,Y)} &&\\
      K\times \uS(X\otimes K, Y)
        \ar@{.>}[rr]^(0.43){\alpha_{K,X}\times\id_{\uS(X\otimes K, Y)}}
        \ar[d]_{\id_K\times \phi_{K,X,Y}}
        && \uS(X,X\otimes K)\times\uS(X\otimes K, Y) \ar@{.>}[d]^{\ub} \\
      K\times \usSet(K,\uS(X,Y)) \ar[rr]_{\ev_{K,\uS(X,Y)}} && \uS(X,Y)}
  \]

(ii) By Lemma \ref{lem:esht.1.16} and the definition of $\phi_{L,X,Y}$,
  the following diagram commutes.
  \[
    \xymatrix{
      K\times \uS(X\otimes L, Y) \ar@{.>}[r]^{u\times\id}
        \ar[d]^{\id_K\times \phi_{L,X,Y}}
        & L\times \uS(X\otimes L, Y)
        \ar@{.>}[rr]^(0.4){\alpha_{L,X}\times\id_{\uS(X\otimes L,Y)}}
        \ar[d]_{\id_L\times \phi_{L,X,Y}}
        && \uS(X,X\otimes L)\times\uS(X\otimes L, Y) \ar@{.>}[dd]^{\ub}\\
      K\times \usSet(L, \uS(X,Y)) \ar[r]^{u\times\id}
        \ar[d]^{\id_K\times \usSet(u,\uS(X,Y))}
        & L\times \usSet(L, \uS(X,Y))
        \ar[drr]^(0.43){\ev_{L,\uS(X,Y)}}
        &&   \\
      K\times \usSet(K,\uS(X,Y)) \ar[rrr]_{\ev_{K,\uS(X,Y)}} &&& \uS(X,Y)}
  \]
  Thus, the dotted arrows computes
  $\sharp(\usSet(u,\uS(X,Y))\bullet\phi_{L,X,Y})$.

(iii)
  Let $k_n\in K_n$ and $f_n\in\uS(X\otimes L,Y)_n$. Then,
  by \textbf{S5} and Lemma \ref{lem:esht.1.25},
  \begin{alignat*}{1}
    \uS(X\otimes u,Y)_n(f_n)\ub_n(\alpha_{K,X})_n(k_n)
    &= f_n\ub_n(\sigma^n_0)^*(\widetilde{X\otimes u})
      \ub_n(\alpha_{K,X})_n(k_n)\\
    &= f_n\ub_n(\alpha_{L,X})_n(u_n(k_n)).
  \end{alignat*}
  Thus
  \[
    \xymatrix{
      K\times\uS(X\otimes K,Y) \ar[rr]^(0.4){\alpha_{K,X}\times\id}
        && \uS(X,X\otimes K)\times\uS(X\otimes K,Y) \ar[d]^{\ub} \\
      K\times \uS(X\otimes L, Y)
        \ar[u]^{\id_K\times\uS(X\otimes u,Y)}
        \ar[d]_{u\times\id} && \uS(X,Y)\\
      L\times\uS(X\otimes L,Y) \ar[rr]_(0.4){\alpha_{L,X}\times\id}
        && \uS(X,X\otimes L)\times\uS(X\otimes L,Y) \ar[u]_{\ub}}
  \]
  is a commutative diagram.

  Therefore, by (i), (ii) and (iii), \eqref{eqn:esht.1.4} holds.

(b) $\phi_{K,-,Y}$:
  Let $u\in\S(W,X)$.
  By Lemma \ref{lem:esht.1.18},
  we need to show that
  \begin{equation}
  \label{eqn:esht.1.5}
    \sharp(\phi_{K,W,Y}\bullet\uS(u\otimes K,Y))
    =\sharp(\usSet(K,\uS(u,Y))\bullet\phi_{K,X,Y})
  \end{equation}
  holds where $\sharp=\sharp_{\uS(X\otimes K,Y),K,\uS(W,Y)}$.
  \[
    \xymatrix{
      \uS(W\otimes K,Y) \ar[rr]^{\phi_{K,W,Y}}
        && \usSet(K,\uS(W,Y))\\
      \uS(X\otimes K,Y) \ar[rr]_{\phi_{K,X,Y}} \ar[u]^{\uS(u\otimes K,Y)} &
        & \usSet(K,\uS(X,Y)) \ar[u]_{\usSet(K,\uS(u,Y))}}
  \]

(i)
  By the definition of $\sharp$ and $\phi_{K,W,Y}$,
  the dotted arrows in the following diagram computes
  $\sharp(\phi_{K,W,Y}\bullet\uS(u\otimes K,Y))$.
  \[
    \xymatrix{
      K\times \uS(X\otimes K, Y)
        \ar@{.>}[d]_{\id_K\times\uS(u\otimes K,Y)} &&\\
      K\times \uS(W\otimes K, Y)
        \ar@{.>}[rr]^(0.43){\alpha_{K,W}\times\id_{\uS(W\otimes K, Y)}}
        \ar[d]_{\id_K\times \phi_{K,W,Y}}
        && \uS(W,W\otimes K)\times\uS(W\otimes K, Y) \ar@{.>}[d]^{\ub} \\
      K\times \usSet(K,\uS(W,Y)) \ar[rr]_{\ev_{K,\uS(W,Y)}} && \uS(W,Y)}
  \]

(ii)
  By the definition of $\phi_{K,X,Y}$ and Lemma \ref{lem:esht.1.15},
  the following diagram commutes.
  \[
    \xymatrix{
      K\times \uS(X\otimes K, Y)
        \ar@{.>}[rr]^(0.43){\alpha_{K,X}\times\id_{\uS(X\otimes K, Y)}}
        \ar[d]_{\id_K\times\phi_{K,X,Y}}  &
        & \uS(X,X\otimes K)\times\uS(X\otimes K, Y) \ar@{.>}[d]^{\ub}\\
      K\times \usSet(K,\uS(X,Y))
        \ar[rr]^{\ev_{K,\uS(X,Y)}}
        \ar[d]_{\id_K\times \usSet(K,\uS(u,Y))}
        && \uS(X,Y) \ar@{.>}[d]^{\uS(u,Y)} \\
      K\times \usSet(K,\uS(W,Y)) \ar[rr]_{\ev_{K,\uS(W,Y)}} && \uS(W,Y)}
  \]
  Then, the dotted arrows computes
  $\sharp(\usSet(K,\uS(u,Y))\bullet\phi_{K,X,Y})$.

(iii)
  Let $k_n\in K_n$ and $f_n\in\uS(X\otimes K,Y)_n$.
  By \textbf{S5} and Lemma \ref{lem:esht.1.30},
  \begin{alignat*}{1}
    \uS(u,Y)_n(f_n\ub_n(\alpha_{K,X})_n(k_n))
    &= f_n\ub_n(\alpha_{K,X})_n(k_n)\ub_n(\sigma^n_0)^*(\widetilde{u})\\
    &= (f_n\ub_n
      (\sigma^n_0)^*(\widetilde{u\otimes K}))\ub_n(\alpha_{K,W})_n(k_n)\\
    &= \uS(u\otimes K,Y)_n(f_n)\ub_n(\alpha_{K,W})_n(k_n).
  \end{alignat*}
  Thus
  \[
    \xymatrix{
      \uS(X,X\otimes K)\times\uS(X\otimes K,Y)
        \ar[rr]^(0.65){\ub}
        && \uS(X,Y) \ar[d]^{\uS(u,Y)}  \\
      K\times \uS(X\otimes K, Y)
        \ar[u]^{\alpha_{K,X}\times\id_{\uS(X\otimes K,Y)}}
        \ar[d]_{\id_K\times\uS(u\otimes K,Y)} && \uS(W,Y)\\
      K\times\uS(W\otimes K,Y)
        \ar[rr]_(0.43){\alpha_{K,W}\times\id_{\uS(W\otimes K,Y)}}
        && \uS(W,W\otimes K)\times\uS(W\otimes K,Y) \ar[u]_{\ub}}
  \]
  is a  commutative diagram.

  Therefore, by (i), (ii) and (iii), \eqref{eqn:esht.1.5} holds.

(c) $\phi_{K,X,-}$:
  Let $u\in\S(Y,Z)$.
  By Lemma \ref{lem:esht.1.18},
  we need to show that
  \begin{equation}
  \label{eqn:esht.1.6}
    \sharp(\phi_{K,X,Z}\bullet\uS(X\otimes K,u))
    =\sharp(\usSet(K,\uS(X,u))\bullet\phi_{K,X,Y})
  \end{equation}
  holds where $\sharp=\sharp_{\uS(X\otimes K,Y),K,\uS(X,Z)}$.
  \[
    \xymatrix{
      \uS(X\otimes K,Y) \ar[rr]^{\phi_{K,X,Y}} \ar[d]_{\uS(X\otimes K,u)}
        && \usSet(K,\uS(X,Y)) \ar[d]^{\usSet(K,\uS(X,u))}\\
      \uS(X\otimes K,Z) \ar[rr]_{\phi_{K,X,Z}}&
        & \usSet(K,\uS(X,Z))}
  \]

(i)
  By definition of $\sharp$ and $\phi_{K,W,Y}$,
  the dotted arrows in the following diagram computes
  $\sharp(\phi_{K,X,Z}\bullet\uS(X\otimes K,u))$.
  \[
    \xymatrix{
      K\times \uS(X\otimes K, Y)
        \ar@{.>}[d]_{\id_K\times\uS(X\otimes K,u)} &&\\
      K\times \uS(X\otimes K, Z)
        \ar@{.>}[rr]^(0.43){\alpha_{K,X}\times\id_{\uS(X\otimes K, Z)}}
        \ar[d]_{\id_K\times \phi_{K,X,Z}}
        && \uS(X,X\otimes K)\times\uS(X\otimes K, Z) \ar@{.>}[d]^{\ub} \\
      K\times \usSet(K,\uS(X,Z)) \ar[rr]_{\ev_{K,\uS(X,Z)}} && \uS(X,Z)}
  \]

(ii)
  By definition of $\phi_{K,X,Y}$ and Lemma \ref{lem:esht.1.15},
  the following is a commutative diagram.
  \[
    \xymatrix{
      K\times \uS(X\otimes K, Y)
        \ar@{.>}[rr]^(0.43){\alpha_{K,X}\times\id_{\uS(X\otimes K, Y)}}
        \ar[d]_{\id_K\times\phi_{K,X,Y}}  &
        & \uS(X,X\otimes K)\times\uS(X\otimes K, Y) \ar@{.>}[d]^{\ub}\\
      K\times \usSet(K,\uS(X,Y))
        \ar[rr]^{\ev_{K,\uS(X,Y)}}
        \ar[d]_{\id_K\times \usSet(K,\uS(X,u))}
        && \uS(X,Y) \ar@{.>}[d]^{\uS(X,u)} \\
      K\times \usSet(K,\uS(X,Z)) \ar[rr]_{\ev_{K,\uS(X,Z)}} && \uS(X,Z)}
  \]
  Then, the dotted arrows computes
  $\sharp(\usSet(K,\uS(X,u))\bullet\phi_{K,X,Y})$.

(iii)
  Let $k_n\in K_n$ and $f_n\in\uS(X\otimes K,Y)_n$.
  By \textbf{S5},
  \begin{alignat*}{1}
    \uS(X,u)_n((f_n)\ub_n(\alpha_{K,X})_n(k_n))
    &= (\sigma^n_0)^*(\widetilde{u})\ub_n(f_n\ub_n(\alpha_{K,X})_n(k_n))\\
    &= ((\sigma^n_0)^*(\widetilde{u})\ub_nf_n)\ub_n(\alpha_{K,X})_n(k_n)\\
    &= \uS(X\otimes K,u)_n(f_n)\ub_n(\alpha_{K,X})_n(k_n).
  \end{alignat*}
  Then,
  \[
    \xymatrix{
      \uS(X,X\otimes K)\times\uS(X\otimes K,Y)
        \ar[rr]^(0.65){\ub}
        && \uS(X,Y) \ar[d]^{\uS(X,u)}  \\
      K\times \uS(X\otimes K, Y)
        \ar[u]^{\alpha_{K,X}\times\id_{\uS(X\otimes K,Y)}}
        \ar[d]_{\id_K\times\uS(X\otimes K,u)} && \uS(X,Z)\\
      K\times\uS(X\otimes K,Z)
        \ar[rr]_(0.43){\alpha_{K,X}\times\id_{\uS(X\otimes K,Z)}}
        && \uS(X,X\otimes K)\times\uS(X\otimes K,Z) \ar[u]_{\ub}}
  \]
  is a  commutative diagram.

  Therefore, by (i), (ii) and (iii), \eqref{eqn:esht.1.6} holds.

(2)Let $K,L\in\sSet$. Let $W,X,Y,Z\in\S$.

(a)
  $\psi_{-,Y,X}$: Let $u\in\sSet(K,L)$.
  By Lemma \ref{lem:esht.1.18},
  we need to show that
  \begin{equation}
  \label{eqn:esht.1.7}
    \sharp(\psi_{K,Y,X}\bullet\uS(Y,X^u))
    =\sharp(\usSet(u,\uS(Y,X))\bullet\psi_{L,Y,X})
  \end{equation}
  holds where $\sharp=\sharp_{\uS(Y,X^L),K,\uS(Y,X)}$.
  \[
    \xymatrix{
      \uS(Y,X^K) \ar[rr]^{\psi_{K,Y,X}}
        && \usSet(K,\uS(Y,X))\\
      \uS(Y,X^L) \ar[rr]_{\psi_{L,Y,X}} \ar[u]^{\uS(Y,X^u)} &
        & \usSet(L,\uS(Y,X)) \ar[u]_{\usSet(u,\uS(Y,X))}}
  \]

(i)
  By the definition of $\sharp$ and $\psi_{K,Y,X}$,
  the dotted arrows of the following diagram computes
  $\sharp(\psi_{K,Y,X}\bullet\uS(Y,X^u))$.
  \[
    \xymatrix{
      K\times \uS(Y,X^L)
        \ar@{.>}[d]_{\id_K\times\uS(Y,X^u)} &&\\
      K\times \uS(Y,X^K)
        \ar@{.>}[rr]^(0.43){(\pr_2,\beta_{K,X}\cdot\pr_1)}
        \ar[d]_{\id_K\times \psi_{K,Y,X}}
        && \uS(Y,X^K)\times\uS(X^K, X) \ar@{.>}[d]^{\ub} \\
      K\times \usSet(K,\uS(Y,X)) \ar[rr]_{\ev_{K,\uS(Y,X)}} && \uS(Y,X)}
  \]

(ii) By Lemma \ref{lem:esht.1.16} and the definition of $\psi_{L,Y,X}$,
  the following diagram commutes.
  \[
    \xymatrix{
      K\times \uS(Y,X^L) \ar@{.>}[r]^{u\times\id}
        \ar[d]^{\id_K\times \psi_{L,Y,X}}
        & L\times \uS(Y,X^L)
        \ar@{.>}[rr]^(0.4){(\pr_2,\beta_{L,X}\cdot\pr_1)}
        \ar[d]_{\id_L\times \psi_{L,Y,X}}
        && \uS(Y,X^L)\times\uS(X^L, X) \ar@{.>}[dd]^{\ub}\\
      K\times \usSet(L, \uS(Y,X)) \ar[r]^{u\times\id}
        \ar[d]^{\id_K\times \usSet(u,\uS(Y,X))}
        & L\times \usSet(L, \uS(Y,X))
        \ar[drr]^(0.43){\ev_{L,\uS(Y,X)}}
        &&   \\
      K\times \usSet(K,\uS(Y,X)) \ar[rrr]_{\ev_{K,\uS(Y,X)}} &&& \uS(Y,X)}
  \]
  Thus, the dotted arrows computes
  $\sharp(\usSet(u,\uS(Y,X))\bullet\psi_{L,Y,X})$.

(iii)
  Let $k_n\in K_n$ and $f_n\in\uS(Y,X^L)_n$. Then,
  by \textbf{S5} and Lemma \ref{lem:esht.1.32},
  \begin{alignat*}{1}
    (\beta_{K,X})_n(k_n)\ub_n \uS(Y,X^u)_n(f_n)
    &= (\beta_{K,X})_n(k_n)\ub_n(\sigma^n_0)^*(\widetilde{X^u})\ub_n
      f_n\\
    &= (\beta_{L,X})_n(u_n(k_n))\ub_nf_n.
  \end{alignat*}
  Thus
  \[
    \xymatrix{
      K\times\uS(Y,X^K) \ar[rr]^(0.42){(\pr_2,\beta_{K,X}\cdot\pr_1)}
        && \uS(Y,X^K)\times\uS(X^K,X) \ar[d]^{\ub} \\
      K\times \uS(Y,X^L)
        \ar[u]^{\id_K\times\uS(Y,X^u)}
        \ar[d]_{u\times\id} && \uS(X,Y)\\
      L\times\uS(Y,X^L) \ar[rr]_(0.42){(\pr_2,\beta_{L,X}\cdot\pr_1)}
        && \uS(Y,X^L)\times\uS(X^L,X) \ar[u]_{\ub}}
  \]
  is a commutative diagram.

  Therefore, by (i), (ii) and (iii), \eqref{eqn:esht.1.7} holds.

(b) $\psi_{K,-,X}$:
  Let $u\in\S(Y,Z)$.
  By Lemma \ref{lem:esht.1.18},
  we need to show that
  \begin{equation}
  \label{eqn:esht.1.8}
    \sharp(\psi_{K,Y,X}\bullet\uS(u,X^K))
    =\sharp(\usSet(K,\uS(u,X))\bullet\psi_{K,Z,X})
  \end{equation}
  holds where $\sharp=\sharp_{\uS(Z,X^K),K,\uS(Y,X)}$.
  \[
    \xymatrix{
      \uS(Y,X^K) \ar[rr]^(0.4){\psi_{K,Y,X}}
        && \usSet(K,\uS(Y,X))\\
      \uS(Z,X^K) \ar[rr]_(0.4){\psi_{K,Z,X}} \ar[u]^{\uS(u,X^K)} &
        & \usSet(K,\uS(Z,X)) \ar[u]_{\usSet(K,\uS(u,X))}}
  \]

(i)
  By the definition of $\sharp$ and $\psi_{K,Y,X}$,
  the dotted arrows in the following diagram computes
  $\sharp(\psi_{K,Y,X}\bullet\uS(u,X^K))$.
  \[
    \xymatrix{
      K\times \uS(Z,X^K)
        \ar@{.>}[d]_{\id_K\times\uS(u,X^K)} &&\\
      K\times \uS(Y,X^K)
        \ar@{.>}[rr]^(0.43){(\pr_2,\beta_{K,X}\cdot\pr_1)}
        \ar[d]_{\id_K\times \psi_{K,Y,X}}
        && \uS(Y,X^K)\times\uS(X^K, X) \ar@{.>}[d]^{\ub} \\
      K\times \usSet(K,\uS(Y,X)) \ar[rr]_{\ev_{K,\uS(Y,X)}} && \uS(Y,X)}
  \]

(ii)
  By the definition of $\psi_{K,Z,X}$ and Lemma \ref{lem:esht.1.15},
  the following diagram commutes.
  \[
    \xymatrix{
      K\times \uS(Z,X^K)
        \ar@{.>}[rr]^(0.42){(\pr_2,\beta_{K,X}\cdot\pr_1)}
        \ar[d]_{\id_K\times\psi_{K,Z,X}}  &
        & \uS(Z,X^K)\times\uS(X^K,X) \ar@{.>}[d]^{\ub}\\
      K\times \usSet(K,\uS(Z,X))
        \ar[rr]^{\ev_{K,\uS(Z,X)}}
        \ar[d]_{\id_K\times \usSet(K,\uS(u,X))}
        && \uS(Z,X) \ar@{.>}[d]^{\uS(u,X)} \\
      K\times \usSet(K,\uS(Y,X)) \ar[rr]_{\ev_{K,\uS(Y,X)}} && \uS(Y,X)}
  \]
  Then, the dotted arrows computes
  $\sharp(\usSet(K,\uS(u,X))\bullet\psi_{K,Z,X})$.

(iii)
  Let $k_n\in K_n$ and $f_n\in\uS(Z,X^K)_n$.
  By \textbf{S5},
  \begin{alignat*}{1}
    (\beta_{K,X})_n(k_n)\ub_n \uS(u,X^K)_n(f_n)
    &= (\beta_{K,X})_n(k_n)\ub_n (f_n \ub_n(\sigma^n_0)^*(\widetilde{u}))\\
    &= ((\beta_{K,X})_n(k_n)\ub_n f_n) \ub_n(\sigma^n_0)^*(\widetilde{u})\\
    &= \uS(u,X)_n((\beta_{K,X})_n(k_n)\ub_n f_n).
  \end{alignat*}
  Thus
  \[
    \xymatrix{
      \uS(Z,X^K)\times\uS(X^K,X)
        \ar[rr]^(0.65){\ub}
        && \uS(Z,X) \ar[d]^{\uS(u,X)}  \\
      K\times \uS(Z,X^K)
        \ar[u]^{(\pr_2,\beta_{K,X}\cdot\pr_1)}
        \ar[d]_{\id_K\times\uS(u,X^K)} && \uS(Y,X)\\
      K\times\uS(Y,X^K)
        \ar[rr]_(0.43){(\pr_2,\beta_{K,X}\cdot\pr_1)}
        && \uS(Y,X^K)\times\uS(X^K,X) \ar[u]_{\ub}}
  \]
  is a  commutative diagram.

  Therefore, by (i), (ii) and (iii), \eqref{eqn:esht.1.8} holds.

(c) $\psi_{K,Y,-}$:
  Let $u\in\S(W,X)$.
  By Lemma \ref{lem:esht.1.18},
  we need to show that
  \begin{equation}
  \label{eqn:esht.1.9}
    \sharp(\psi_{K,Y,X}\bullet\uS(Y,u^K))
    =\sharp(\usSet(K,\uS(Y,u))\bullet\psi_{K,Y,W})
  \end{equation}
  holds where $\sharp=\sharp_{\uS(Y,W^K),K,\uS(Y,X)}$.
  \[
    \xymatrix{
      \uS(Y,W^K) \ar[rr]^(0.4){\psi_{K,Y,W}} \ar[d]_{\uS(Y,u^K)}
        && \usSet(K,\uS(Y,W)) \ar[d]^{\usSet(K,\uS(Y,u))}\\
      \uS(Y,X^K) \ar[rr]_(0.4){\psi_{K,Y,X}}&
        & \usSet(K,\uS(Y,X))}
  \]

(i)
  By definition of $\sharp$ and $\psi_{K,Y,X}$,
  the dotted arrows in the following diagram computes
  $\sharp(\psi_{K,Y,X}\bullet\uS(Y,u^K))$.
  \[
    \xymatrix{
      K\times \uS(Y,W^K)
        \ar@{.>}[d]_{\id_K\times\uS(Y,u^K)} &&\\
      K\times \uS(Y,X^K)
        \ar@{.>}[rr]^(0.43){(\pr_2,\beta_{K,X}\cdot\pr_1)}
        \ar[d]_{\id_K\times \psi_{K,Y,X}}
        && \uS(Y,X^K)\times\uS(X^K,X) \ar@{.>}[d]^{\ub} \\
      K\times \usSet(K,\uS(Y,X)) \ar[rr]_{\ev_{K,\uS(Y,X)}} && \uS(Y,X)}
  \]

(ii)
  By definition of $\psi_{K,Y,W}$ and Lemma \ref{lem:esht.1.15},
  the following is a commutative diagram.
  \[
    \xymatrix{
      K\times \uS(Y,W^K)
        \ar@{.>}[rr]^(0.43){(\pr_2,\beta_{K,W}\cdot\pr_1)}
        \ar[d]_{\id_K\times\psi_{K,Y,W}}  &
        & \uS(Y,W^K)\times\uS(W^K, W) \ar@{.>}[d]^{\ub}\\
      K\times \usSet(K,\uS(Y,W))
        \ar[rr]^{\ev_{K,\uS(Y,W)}}
        \ar[d]_{\id_K\times \usSet(K,\uS(Y,u))}
        && \uS(Y,W) \ar@{.>}[d]^{\uS(Y,u)} \\
      K\times \usSet(K,\uS(Y,X)) \ar[rr]_{\ev_{K,\uS(Y,X)}} && \uS(Y,X)}
  \]
  Then, the dotted arrows computes
  $\sharp(\usSet(K,\uS(Y,u))\bullet\psi_{K,Y,W})$.

(iii)
  Let $k_n\in K_n$ and $f_n\in\uS(Y,W^K)_n$.
  By \textbf{S5} and Lemma \ref{lem:esht.1.36},
  \begin{alignat*}{1}
    \uS(Y,u)_n((\beta_{K,W})_n(k_n)\ub_n f_n)
    &= (\sigma^n_0)^*(\widetilde{u})\ub_n((\beta_{K,W})_n(k_n)\ub_n f_n)\\
    &= (\beta_{K,X})_n(k_n)\ub_n(\sigma^n_0)^*(\widetilde{u^K})\ub_n f_n\\
    &= (\beta_{K,X})_n(k_n)\ub_n\uS(Y,u^K)_n(f_n).
  \end{alignat*}
  Then,
  \[
    \xymatrix{
      \uS(Y,W^K)\times\uS(W^K,W)
        \ar[rr]^(0.65){\ub}
        && \uS(Y,W) \ar[d]^{\uS(X,u)}  \\
      K\times \uS(Y,W^K)
        \ar[u]^{(\pr_2,\beta_{K,W}\cdot\pr_1)}
        \ar[d]_{\id_K\times\uS(Y,u^K)} && \uS(Y,X)\\
      K\times\uS(Y,X^K)
        \ar[rr]_(0.43){\alpha_{K,X}\times\id_{\uS(X\otimes K,Z)}}
        && \uS(Y,X^K)\times\uS(X^K,X) \ar[u]_{\ub}}
  \]
  is a  commutative diagram.

  Therefore, by (i), (ii) and (iii), \eqref{eqn:esht.1.9} holds.

(3) follows from (1) and (2).
\end{proof}

\end{document}